\documentclass[11pt]{article}
\usepackage[utf8]{inputenc}

\usepackage{
amsmath,amssymb,amsfonts,amsthm} 
\usepackage{bm}
\usepackage{colonequals}
\usepackage{comment}
\usepackage{epsfig} \usepackage{latexsym,nicefrac,bbm}
\usepackage
{xspace}
\usepackage{color,fancybox,graphicx,subfigure,fullpage}
\usepackage[top=1in, bottom=1in, left=1in, right=1in]{geometry}
\usepackage{tabularx} 
\usepackage{hyperref} 
\usepackage{cleveref}
\usepackage{pdfsync}
\usepackage[boxruled,linesnumbered]{algorithm2e}
\usepackage{siunitx}

\usepackage{multicol}
\usepackage{enumitem}
\usepackage{float}
\usepackage{tikz}
\usepackage{tabularray}
\usepackage{framed}

\makeatletter
\newcommand{\vast}{\bBigg@{4}}
\newcommand{\Vast}{\bBigg@{5}}
\makeatother

\newtheorem{theorem}{Theorem}[section]

\newtheorem{conjecture}[theorem]{Conjecture}
\newtheorem{definition}[theorem]{Definition}
\newtheorem{lemma}[theorem]{Lemma}
\newtheorem{remark}[theorem]{Remark}
\newtheorem{proposition}[theorem]{Proposition}

\newcommand{\bN}{\mathbb{N}}

\newcommand{\bR}{\mathbb{R}}

\newcommand{\bZ}{\mathbb{Z}}
\DeclareSymbolFont{bbold}{U}{bbold}{m}{n}
\DeclareSymbolFontAlphabet{\mathbbold}{bbold}
\newcommand{\One}{\mathbbold{1}}

\DeclareMathOperator*{\argmax}{arg\,max}

\renewcommand{\epsilon}{\varepsilon}
\renewcommand{\Tilde}{\widetilde}
\renewcommand{\hat}{\widehat}

\newcommand{\Ex}{\mathop{\mathbb{E}}}

\newcommand{\eps}{\varepsilon}
\renewcommand{\epsilon}{\varepsilon}

\def\abs#1{\left|#1  \right|}

\def\norm#1{\left\| #1 \right\|}

\newcommand\zzero{\boldsymbol{0}}

\def\calN{\mathcal{N}}

\renewcommand\AA{\boldsymbol{\mathit{A}}}
\newcommand\BB{\boldsymbol{\mathit{B}}}

\newcommand\GG{\boldsymbol{\mathit{G}}}

\newcommand\II{\boldsymbol{\mathit{I}}}

\newcommand\LL{\boldsymbol{\mathit{L}}}
\newcommand\MM{\boldsymbol{\mathit{M}}}

\newcommand\PP{\boldsymbol{\mathit{P}}}
\newcommand\QQ{\boldsymbol{\mathit{Q}}}
\newcommand\RR{\boldsymbol{\mathit{R}}}
\renewcommand\SS{\boldsymbol{\mathit{S}}}

\newcommand\UU{\boldsymbol{\mathit{U}}}
\newcommand\VV{\boldsymbol{\mathit{V}}}

\newcommand\XX{\boldsymbol{\mathit{X}}}
\newcommand\YY{\boldsymbol{\mathit{Y}}}

\newcommand\bb{\boldsymbol{\mathit{b}}}

\newcommand\ee{\boldsymbol{\mathit{e}}}

\renewcommand\gg{\boldsymbol{\mathit{g}}}

\renewcommand\ll{\boldsymbol{\mathit{l}}}

\newcommand\qq{\boldsymbol{\mathit{q}}}

\newcommand\rr{\boldsymbol{\mathit{r}}}

\newcommand\vv{\boldsymbol{\mathit{v}}}

\newcommand\xx{\boldsymbol{\mathit{x}}}

\newcommand\yy{\boldsymbol{\mathit{y}}}

\newcommand\mm{\boldsymbol{\mathit{m}}}

\newcommand\LLhat{\boldsymbol{\widehat{\mathit{L}}}}

\newcommand\PPhat{\boldsymbol{\widehat{\mathit{P}}}}

\newcommand\UUhat{\boldsymbol{\widehat{\mathit{U}}}}

\newcommand\ssigma{\boldsymbol{\sigma}}

\newcommand{\PPi}{\boldsymbol{\Pi}}
\newcommand{\SSigma}{\boldsymbol{\Sigma}}

\newdimen\pIR
\newcommand\StevesR{{\rm I\kern\pIR R}}

\def\bvec#1{\mathbf{#1}}

\def\defeq{\stackrel{\mathrm{def}}{=}}

\newsavebox{\Codefig}

\def\lprod{{\ooalign{\hfil\raise .00ex\hbox{\scriptsize L}\hfil\crcr\mathhexbox20D}}}

\newcommand\numberthis{\addtocounter{equation}{1}\tag{\theequation}}

\usepackage{authblk}

\usepackage[textsize=tiny]{todonotes}

\title{Large Growth Happens:  Gaussian Elimination with Partial Pivoting on Random Matrices
\thanks{This research was partially supported by 
	ONR Award N00014-24-12611.}}
\author[1,2]{Daniel A.\ Spielman\thanks{Email: \texttt{daniel.spielman@yale.edu}.}}
\author[1]{Xifan Yu\thanks{Email: \texttt{xifan.yu@yale.edu}.}}
\affil[1]{Department of Computer Science, Yale University}
\affil[2]{Department of Statistics and Data Science, Yale University}

\date{\today}

\begin{document}

\maketitle

\begin{abstract}
    We prove that the probability that Gaussian elimination with partial pivoting on $n \times n$ Gaussian random matrices has growth $\rho$ is at least inverse quasi-polynomial: $\Omega\left(\exp\left(-c \log^2 \left(\frac{\rho}{\sqrt{n}}\right)\log(n)\right)\right)$ for some constant $c > 0$.
    In particular, the probability of growth $\rho \ge n^{1/2 + \epsilon}$ is at least $\exp(- O (\epsilon^2 \log^3 (n) ))$.
    This lower bound breaks standard conjectures in the smoothed and average-case analysis of Gaussian elimination.
    To the best of our knowledge, it is the first non-trivial lower bound on the probability of large growth for Gaussian random matrices.
\end{abstract}

\thispagestyle{empty}
\clearpage
\pagenumbering{arabic}

\allowdisplaybreaks{

\section{Introduction}

Despite the fact that it performs poorly in the worst case, Gaussian\footnote{While we use the name Gaussian elimination, the algorithm is much older. The first reference of which we are aware appeared in the Chinese text ``The Nine Chapters on the Mathematical Arts''.}
elimination with partial pivoting (GEPP) is the most commonly used algorithm to solve systems of linear equations in general dense matrices. 
Urschel~\cite{Urschel2025} writes 
\begin{quote}
This procedure is known to
be horribly unstable in the worst-case. Still,
Gaussian elimination with partial pivoting remains the
most popular method for solving an unstructured linear
system numerically: the backslash ``\textbackslash'' command in Matlab
and Julia, and the linalg.solve() command in Python, all
employ this method (by calling LAPACK routines) when
faced with an arbitrary matrix.
\end{quote}

Researchers have tried to explain the success of GEPP by proving that it performs well on particular classes of matrices~\cite{WilkinsonBook}, by extensive experimentation and reports of empirical success, and by experimental, heuristic, and rigorous analysis of its behavior on random matrices~\cite{SchreiberTrefethen,TrefethenBau,DriscollMaki,HuangTikhomirov}, and on slight random perturbations of arbitrary matrices~\cite{Sankar}. 

The way that GEPP fails is that it can produce highly inaccurate solutions. 
Since the foundational work of Wilkinson~\cite{wilkinson1961error, WilkinsonBook}, it has been known that the relative error of the solution produced on $n$-dimensional systems is bounded by the product of machine precision, a polynomial in $n$, a condition number of the matrix, and the growth factor of GEPP on the matrix.
While the first three terms are essentially unavoidable for every algorithm, the growth factor is where GEPP can fall short.
It can be as large as $2^{n-1}$ for $n$-dimensional matrices.
For analysis based on growth factors, 
this could be disastrous because standard machine precision is around $10^{-16}$, which we could only prove to be enough to solve systems in around 50 dimensions. 
While one could in theory use higher precision arithmetic, using $\Omega(n)$ bits of precision to solve systems of dimension $n$ would be prohibitively slow.
Even though there exist slightly slower variants of Gaussian Elimination that are guaranteed to have small growth\footnote{
The growth of complete pivoting is at most $2^{O(\log^2 n)}$ and was recently shown to be $2^{\Omega(\log^2 n)}$ in the worst case\cite{shah2026entrygrowthgaussianelimination}.}, 
GEPP is favored because it is faster and it is almost never observed to have large growth in practice.  
Algorithms based on the QR-factorization avoid the issue entirely, and only require a logarithmic number of bits.

Experimental explorations of the growth of GEPP on matrices of independent normal entries, beginning with work of MacLeod~\cite{MacLeod} and carried out extensively by Schreiber and Trefethen~\cite{SchreiberTrefethen}, Edelman and Mascarenhas~\cite{EdelmanMascarenhas} and Driscoll and Maki~\cite{DriscollMaki}, led to the conjecture that the probability of growth $\rho$ decreases exponentially in $\rho$. 
Driscoll and Maki fit an exponential function to the density.
Sankar, Spielman, and Teng~\cite{SankarSpielmanTeng} even conjectured this in the (more general) smoothed case: they conjectured that for every matrix $\AA$ of norm one and a standard Gaussian random matrix $\RR$, the probability that $\AA + \sigma \RR$ has growth at least $\rho$ decreases like $n^{O(1)} \exp(-\Omega(\rho^c  \sigma))$, for some constant $c$.  
Trefethen and Bau~\cite[Lecture 22]{TrefethenBau} write ``The density appears to decrease exponentially with $\rho$,'' and Trefethen and Schreiber say ``Experiments leave little doubt that the tail decays exponentially.''
{\it We disprove these conjectures.}

Theorem~\ref{thm:main} says that for standard Gaussian random matrices, the probability of growth at least $\rho$ is at least $\exp\left(-710 \log^2\left(\frac{\rho}{\sqrt{n}}\right)\log(n)\right)$.
    In particular, the probability of growth $\rho = n^{1/2 + \epsilon}$ is at least $\exp(- O (\epsilon^2 \log^3 (n) ))$.
When combined with known upper bounds, described below, this tells us that the tail of $\rho$ is better fit by a log-normal distribution than by an exponential.


Trefethen~\cite{Trefethen2012} asks ``For matrices with independent normally distributed entries, how fast does the probability of encountering a growth factor $> \rho$ decrease as a function of $\rho$? Experiments show that it decreases exponentially, or at least faster than any inverse power of $\rho / n^{1/2}$,'' and then offers \$1,000 for the proof of a conjecture that is both tighter and looser than the disproved conjectures: 
  it is looser in that it only requires a sub-polynomial rate of decay, but it is tighter in that it specifies a sharp bound on the typical growth.
\begin{conjecture}[Trefethen]
For all positive constants $\epsilon$ and $C$, and for all sufficiently large $n$,
  the probability that an $n$-by-$n$ matrix of independent $\calN (0,1)$ entries has 
  growth greater than $n^{1/2 + \epsilon}$ is at most $n^{-C}.$    
\end{conjecture}

In a recent breakthrough, Urschel~\cite{urschel2026growthfactorrandommatrices}
  proved Trefethen's conjecture.
He proves that, for every $\alpha > 0$ there exist constants 
  $C_{\alpha}$ and $n_{\alpha}$, so that for $n > n_{\alpha}$
  the probability $\rho > \sqrt{n} \exp (C_{\alpha} \sqrt{\log n})$
  is at most $n^{-\alpha}$.
Our result is a lower bound showing that Urschel's upper bound is close to tight.
In Urschel's regime, we show that the probability of such growth is at least inverse-quasi-polynomial and in particular at least $n^{-710 C_{\alpha}^2 \log n}.$

We do not claim our results are a reason not to use GEPP, as they do not show a significant probability of large growth. 
Rather, they just show that it is more likely than has been conjectured, and that the upper bounds on the probability of large growth are close to tight.

\subsection*{Role of Artificial Intelligence}

The major ideas for the proofs were proposed by the authors.
In response to a request to compare the two probabilities addressed in Lemma~\ref{lem:conic-prob-transport}, GPT-5.4 extra high produced a proof of a weak form of that lemma, stated in terms of the condition number of the covariance matrix. We greatly simplified that proof and used that lemma to obtain a lower bound of $\exp(-O(\log^3(\rho)))$ for the probability of large growth. 
We later realized that we could strengthen the lemma to produce a bound in terms of the determinant of the covariance matrix, instead of its condition number.
That was the key to our improved lower bound of $\exp(-O(\log^2(\rho/\sqrt{n})\log(n)))$ on the probability that GEPP has large growth.

\section{Preliminaries}\label{sec:prelim}

In this paper, we will use bold letters for matrices and vectors, and non-bold letters for numbers. When we use $E$ to denote a set or event, we use $E^c$ to denote its complement.

For $\AA \in \bR^{n \times n}$, we will use $\|\AA\|_2$ to denote its operator norm, 
$\max_{\xx \neq \bvec{0}} \norm{\AA \xx}_2 / \norm{\xx}_2$.
We use $\kappa(\AA) \colonequals \|\AA\|_2 \|\AA^{-1}\|_2$ to denote its condition number if $\AA$ is non-singular. 
We write $\AA \succeq \BB$ if $\AA - \BB$ is positive semidefinite.

We use Matlab notation to indicate sub-matrices. 
That is, $\AA [i,j]$ is the entry of $\AA$ in row $i$ and column $j$.
$\AA[i,:]$ is the $i$th row of $\AA$ and $\AA[:,j]$ is its $j$th column.
Using $j{:}k$ to write the range from $j$ to $k$, we can write the sub-vector of row $i$ of $\AA$ consisting of entries in columns $j$ through $k$ as $\AA[i, j{:}k]$. Finally $\AA[j{:}k, j{:}k]$ is the sub-matrix consisting of entries in rows and columns between $j$ and $k$. 

Additionally, we introduce the following notation for removing certain entries from vectors and matrices. Let $1 \le i \le n$. For $\vv \in \bR^n$, we will use $\vv[-i] \in \bR^{n-1}$ to denote the $(n-1)$-dimensional vector obtained from $\vv$ by removing entry $i$. For $\AA \in \bR^{n \times n}$, we will use $\AA[-i,-i]$ to denote the $(n-1) \times (n-1)$ matrix obtained from $\AA$ by removing row $i$ and column $i$.

\subsection{Gaussian Elimination with Partial Pivoting}

Let $n \in \bN$, and $\AA \in \bR^{n \times n}$ be a non-singular matrix. Gaussian elimination with partial pivoting (GEPP) applied to $\AA$ constructs a permutation matrix $\PP$, a lower triangular $\LL$, and an upper triangular $\UU$ such that $\PP \AA = \LL\UU$. 
Given this factorization, one can solve systems $\AA \xx = \bb$ by permuting and then solving systems in the triangular matrices $\LL$ and $\UU$, each of which takes time $O(n^2)$.

Gaussian Elimination proceeds by iteratively subtracting multiples of rows of the matrix to convert it into the upper triangular form $\UU$. 
The multiples that are subtracted are encoded in $\LL$.
The naive version of Gaussian Elimination lacks the permutation matrix $\PP$, but can fail if it encounters a zero pivot in which case it would divide by zero.
For numerical computation, small pivots are almost as bad as zero pivots.
So, partial pivoting re-orders the rows so that the entry of the largest absolute value in each column is always the pivot.
In the $k$th iteration, it swaps the $k$th row with the row below that has the largest entry in the $k$th column, and then uses that row to eliminate the entries below it in the $k$th column.

We now give a more formal description of the algorithm.

\begin{enumerate}
    \item Initialize $\AA^{(1)} \colonequals \AA$.
    \item At step $k$ for $1 \le k \le n-1$, select the pivot $i_k \in \{k, k+1, \dots, n\}$ according to the partial pivoting rule:
    \begin{align*}
        i_k \colonequals \argmax_{k \le i \le n}\left|\AA^{(k)}[i,k]\right|.
    \end{align*}
    In the case when the argmax is not unique, break ties arbitrarily. Let $\PP^{(k)}$ denote the permutation matrix that swaps the $k$th row and the $i_k$-th row, and $\BB^{(k)} \colonequals \PP^{(k)} \AA^{(k)}$. Define the multiplier vector $\ll^{(k)} \in [-1,1]^n$ by
    \begin{align*}
        \ll^{(k)}[i] \colonequals \begin{cases}
            \frac{\BB^{(k)}[i,k]}{\BB^{(k)}[k, k]} & \quad \text{ if } k < i ,\\
            0 & \quad \text{ otherwise},
        \end{cases}
    \end{align*}
    and $\bb^{(k)} \in \bR^n$ to be the vector of $k$th row of $\BB^{(k)}$ as
    \begin{align*}
        \bb^{(k)}[i] \colonequals \begin{cases}
            \BB^{(k)}[k,i] & \quad \text{ if } k \le i,\\
            0 & \quad \text{ otherwise}.
        \end{cases}
    \end{align*}
    Now set
    \begin{align*}
        \AA^{(k+1)} \colonequals \BB^{(k)} - \ll^{(k)} (\bb^{(k)})^\top,
    \end{align*}
    or equivalently,
    \begin{align*}
        \AA^{(k+1)} \colonequals (\II_n - \ll^{(k)}\ee_k^\top)\BB^{(k)}.
    \end{align*}
    \item When the algorithm terminates after step $n-1$, construct the permutation matrix $P$ by
    \begin{align*}
        \PP &\colonequals \PP^{(n-1)}\PP^{(n-2)} \dots \PP^{(2)} \PP^{(1)},
    \end{align*}
    and triangular $\LL, \UU \in \bR^{n \times n}$ by
    \begin{align*}
        \LL[i,j] &= \begin{cases}
            1 & \quad \text{ if } i = j,\\
            \hat{\ll}^{(j)}[i] & \quad \text{ if } i > j,\\
            0 & \quad \text{ otherwise, }
        \end{cases}\\
        \UU[i,j] &= \begin{cases}
            \bb^{(i)}[j] & \quad \text{ if } j \ge i,\\
            0 & \quad \text{ otherwise},
        \end{cases}
    \end{align*}
    where we denote $
    \bb^{(n)}[n] \colonequals A^{(n)}[n,n]$, and the permuted multiplier vectors $\hat{\ll}^{(k)}$ by
    \begin{align*}
        \hat{\ll}^{(k)} \colonequals \PP^{(n-1)}\PP^{(n-2)}\dots \PP^{(k+2)}\PP^{(k+1)}\ll^{(k)}, \quad \text{for } 1\le k \le n-1.
    \end{align*}
\end{enumerate}

The numerical error of GEPP (and Gaussian elimination in general) is closely related to the growth factor, defined as follows.
\begin{definition}[Growth Factor]
    The growth factor of $\AA$ is
    \begin{align*}
        \rho(\AA) = \frac{\max_{i,j} \left|\UU[i,j]\right|}{\max_{i,j} \left|\AA[i,j]\right|},
    \end{align*}
where $\UU$ is as computed under exact arithmetic.
\end{definition}

There are many results relating growth and condition numbers to the accuracy of the solution of systems of linear equations that vary in the norms and condition numbers used~\cite{wilkinson1961error,higham2002accuracy,Urschel2025}.
A typical result says that the maximum difference between the entries of $\AA$ and the product $\LLhat \UUhat$, where $\LLhat$ and $\UUhat$ are the matrices computed with machine precision $\eps$, is at most $O(n^3 \eps \rho(\AA) \norm{\AA}_{\infty})$ \cite[Theorem 9.4]{higham2002accuracy}.
This can, in turn, be used to upper bound the relative error of the solution obtained to the system $\AA \xx = \bb$ by applying forward and backward substitution in $\LLhat$ and $\UUhat$.

\subsection{Main Result: lower bound on probability of GEPP having large growth}

We establish the following lower bound on the probability that GEPP has large growth factor on a matrix with independent standard normal entries.

\begin{theorem}\label{thm:main}
    Let $n \ge 122$, and $\AA$ be an $n \times n$ matrix with independent $\calN(0,1)$ entries. Then, for any 
    $$\rho \in \left[100\sqrt{n}\log(n),\; \frac{\left(\frac{7}{5}\right)^{\frac{n-22}{2}}}{3\sqrt{42} (n-22)\sqrt{\log(n)}}\right],$$
    we have
    \begin{align*}
        \Pr\left(\rho(\AA) \ge \rho\right) &\ge \exp\left(-710 \log^2\left(\frac{\rho }{\sqrt{n}}\right)\log(n)\right).
    \end{align*}
\end{theorem}

Note that this bound applies for a very large range of $\rho$: from values on the order of $\sqrt{n}\log(n)$ to values that are exponential in $n$.

\begin{remark}
    We have not attempted to optimize the constants appearing in our theorem. One could improve the constant by changing the constant $\frac{3}{2}$ in Proposition~\ref{prop:covariance-growth} and the corresponding constant in Proposition~\ref{prop:suff-cond-norm-increase} and in Proposition~\ref{prop:isotropic-prob}. One could also get a better constant by choosing a larger parameter $T$ and looking at larger values of $n$ and $\rho$ used by the proof of Theorem~\ref{thm:main} in Section~\ref{sec:proof-main}.
\end{remark}

If our goal were to prove that GEPP is accurate, then we would be concerned that we consider the magnitude of entries in the $\UU$ that is obtained by exact arithmetic, rather than by the $\UUhat$ obtained by arithmetic in bounded precision. 
However 
  there is no reason to believe that using bounded precision would aid accuracy.
Moreover, if $\UU$ is large but $\UUhat$ is small then either $\PPhat$ differs from $\PP$ or $\UU$ and $\UUhat$ must differ significantly and $\LLhat \UUhat$ can be far from $\PP \AA$.

\subsection{Prior Work}

The first significant upper bound on the growth of GEPP on random matrices was obtained by Sankar~\cite{Sankar} in the context of smoothed analysis: Sankar proved that the probability that a normal perturbation of variance $\sigma^2$ of an arbitrary matrix of norm at most $1$ has growth $\rho$ decreases at a rate proportional to $(\rho \sigma)^{-\Omega(\log n)},$ after a quasi-polynomial typical value.
For the case of matrices of independent $\calN (0,1)$ variables, Sankar's result says that the probability of growth exceeding $\rho$ is at most 
\[
n^{-c_1 \log(\rho / n^{c_2 \log n})},
\]
for some constants $c_1$ and $c_2$.
Thus, Sankar's bound satisfies the rate of Trefethen's conjecture, but is off by a quasi-polynomial term in the typical value.
It only guarantees that growth is unlikely to be larger than $n^{O(\log n)}$.
But it does show that the probability of growth larger than that decreases at an inverse-quasipolynomial rate.
Our lower bound shows that one can not prove a rate of decay better than inverse-quasipolynomial, although the exact bounds do not match.
One caveat to Sankar's result is that it is for the LU factorization obtained in exact arithmetic, and there is some chance that the growth obtained by floating point computation differs.

Huang and Tikhomirov~\cite{HuangTikhomirov} trade a better typical value for a worse decay rate, while proving bounds that hold under exact arithmetic and under floating point computation.
They also address both the matrix $\UU$ that would be computed in exact arithmetic, and the matrix $\UUhat$ that is computed in floating point arithmetic.
For exact arithmetic, they prove that there exists a constant $C > 1$ so that for all $p \geq 1$ there exists a $n_p$ so that for $n > n_p$, the probability of growth exceeding $n^t$ is at most $5 n^{-p t}$, provided that 
  $t \geq C p^2.$
  
This was the first result showing that the growth of GEPP is at most $n^{O(1)}$ with high probability.
Beyond that, it decays at a sub-inverse polynomial rate.
For any rate $p \geq 1$, it says that the probability of growth $\rho$ decays like $\rho^{-p}$ after 
   $\rho \geq \Omega(n^{C p^2})$.
They then leverage this result to show that the growth of the matrix $\UUhat$ satisfies a similar, although slightly weaker, upper bound.
While one might be tempted to read this bound with $t = p^2$ and allow both to grow with $n$, the form of their bound does not allow that as $n_p$ grows with $p$ in a way that we do not yet fully understand. 

Urschel's theorem~\cite{urschel2026growthfactorrandommatrices} 
  obtains the right order of growth of GEPP on typical matrices, showing that it is at most $n^{1/2 + o(1)}$ with high probability. 
Urschel also improves the smoothed analysis of the growth of GEPP: on a normal perturbation of variance $\sigma^2$ of an arbitrary matrix of norm at most $1$, the growth of GEPP is at most $n^{O(1)} \sigma^{-O(\log\log n)}$ with high probability.

\subsection{Proof Strategy}

We consider the distribution of the last $n-k+1$ entries of the $k$th column of $\AA^{(k)}$.
As the actions of the algorithm before the $k$th step do not depend on the entries in the $k$th column, this vector has a normal distribution conditioned on the realization of the first $k-1$ columns of $\AA$.
Let $\SSigma^{(k)}$ be its covariance matrix. 
We focus on showing that the operator norm of $\SSigma^{(k)}$ can become large for some $k$, and show that when this happens the growth factor is likely to be large as well.

We prove in Proposition~\ref{prop:det-lower-bound} that at the end of the algorithm when $\delta \defeq n-k+1$ is small, the determinant of $\SSigma^{(k)}$ is at least $n^{\delta(1 - o(1))}$ with high probability. This provides a lower bound of $n^{1-o(1)}$ on the operator norm of $\SSigma^{(k)}$. 
Next, we show that there is a reasonable chance that the operator norm of $\SSigma^{(k+1)}$ is at least $7/5$ times the operator norm of $\SSigma^{(k)}$.
This probability has approximately the form $n^{-O(\delta)}$. 
Our proof of this fact uses three ingredients: a) the determinant of $\SSigma^{(k)}$ is at most $n^{O(\delta)}$ with high probability, established in Proposition~\ref{prop:det-upper-bound} and Proposition~\ref{prop:det-det-control}, b) an explicit conic event which is a sufficient condition for the operator norm of $\SSigma^{(k+1)}$ to grow by a factor of $7/5$ compared to the operator norm of $\SSigma^{(k)}$, established in Proposition~\ref{prop:suff-cond-norm-increase}, and c) a comparison inequality for the probability of the conic event under Gaussian measures in Lemma~\ref{lem:conic-prob-transport}. By setting $\delta \approx \log_{7/5} (g)$ and applying this argument in iterations $n-\delta$ through the last, 
we show that with probability approximately $n^{-O(\delta^2)} = \Omega(\exp (-c \log^2(g)\log(n)))$, one of the last $\delta$ covariance matrices probably has operator norm at least $n^{1-o(1)}g$.
This implies that one of the entries in the corresponding column is probably large.

\section{Covariance Update Process}

Let $\MM^{(k)} = \AA^{(k)}[k{:}n, k{:}n]$, and assume that each column of $\MM^{(k)}$ is an independent 
  $\calN(\zzero , \SSigma^{(k)})$ vector.
We will show that each column of $\MM^{(k+1)}$ is also a normal random vector, and we will derive their covariance matrix, 
  $\SSigma^{(k+1)}$.

We begin by re-stating the $k$th elimination step in terms of $\MM^{(k)}$. 
We 
\begin{itemize}
    \item Set $i = \argmax_{i} \abs{\MM^{(k)}[i,1]}$.
    \item Let $\PP_{1,i}$ be the permutation matrix that swaps rows $1$ and $i$.
    \item Set $\BB = \PP_{1,i} \MM^{(k)}$
    \item Set $\ll$ to be the vector such that $\ll[1] = 0$ and $\ll[j] = \BB[j,1] / \BB[1,1]$ for $j > 1$.
    \item Set $\AA^{(k+1)}[k{:}n, k{:}n] = \left( \II - \ll \ee_1^{\top} \right) \BB$.
\end{itemize}

When forming $\MM^{(k+1)}$ we want to omit the first row and column of 
  $\AA^{(k+1)}[k{:}n, k{:}n]$.
So, we let $\PPi_{1}$ be the projection matrix that when multiplied by a column vector removes its first entry.
That is, $\PPi_{1}$ looks like an identity matrix with the first row removed.
We now have
\begin{align*}
 \MM^{(k+1)} 
   & = \PPi_1  (\II - \ll \ee_1^{\top}) \BB \PPi_1^\top \\
  & = \PPi_1  (\II - \ll \ee_1^{\top} ) \PP_{1,i} \MM^{(k)} \PPi_1^\top .
\end{align*}
Since $\PPi_1$ removes the first row of the matrix it multiplies, the first entry of $\ll$ has no impact on this expression.
We find it convenient to work with a vector $\qq$ that agrees with $\ll$, except that it is $1$ in the first coordinate:
$\qq[j] = \BB[j,1] / \BB[1,1]$ for all $j$.
We also define $\rr = \PP_{1,i} \qq$, and recall that $\PP_{1,i}$ is its own inverse.
Continuing our computation, we obtain
\begin{align*}
  \MM^{(k+1)}  & =  \PPi_1  \left( \II - \qq \ee_1^{\top} \right) \PP_{1,i} \MM^{(k)} \PPi_1^\top \\
   & = \PPi_1  (\PP_{1,i}  - \qq \ee_i^{\top} )  \MM^{(k)} \PPi_1^\top, 
     \qquad \text{as $\ee_1^\top \PP_{1,i} = \ee_i^\top$,}\\
  & = \PPi_1  \PP_{1,i} (\II  - \rr \ee_i^{\top} ) \MM^{(k)} \PPi_1^\top.
\end{align*}

As the operator we apply on the left to $\MM^{(k)}$, $\PPi_1  \PP_{1,i} (\II  - \rr \ee_i^{\top} ),$
 only depends on the entries in the first column of $\MM^{(k)}$, and the columns of $\MM^{(k)}$ were independently drawn from 
   $\calN(\zzero , \SSigma^{(k)})$, after conditioning on the first column the remaining columns have distribution $\calN(\zzero , \SSigma^{(k+1)})$
   where 
\[
\SSigma^{(k+1)} = \PPi_1  \PP_{1,i} \left( \II  - \rr \ee_i^{\top} \right) \SSigma^{(k)} 
  \left( \II  - \ee_i \rr^{\top} \right) \PP_{1,i}^{\top} \PPi_1^{\top}.
\]

Finally, we observe that the distribution of $\rr$ has the following simple description in terms of $\MM^{(k)}$:
\[
    i = \argmax_i \left|\MM^{(k)}[i,1] \right|
    \quad \text{and} \quad
    \rr^{(k)} = \frac{\MM^{(k)}[:,1]}{\MM^{(k)}[i,1]}.
\]

We summarize this computation in the following proposition.
\begin{proposition}\label{pro:covariance_update}
Let the columns of $\MM^{(k)}$ be independent $\calN(\zzero, \SSigma^{(k)})$ vectors.
Then, conditioning on the first column $\yy$ of $\MM^{(k)}$, the columns of $\MM^{(k+1)}$ are independent $\calN(\zzero, \SSigma^{(k+1)})$ vectors,
where $\SSigma^{(k+1)}$ may be obtained by 
    \begin{itemize}
    \item sampling $\yy \sim \calN(\zzero, \SSigma^{(k)})$,
    \item setting the pivot $i = \argmax |\yy[i]|$,
    \item defining $\rr = \yy/\yy[i]$,
    \item and setting $\SSigma^{(k+1)} = \PPi_1  \PP_{1,i} \left( \II  - \rr \ee_i^{\top} \right) \SSigma^{(k)} 
  \left( \II  - \ee_i \rr^{\top} \right) \PP_{1,i}^{\top} \PPi_1^{\top}$.
\end{itemize}
\end{proposition}

We will analyze the growth factor of GEPP on Gaussian random matrices through the lens of this covariance update process.

To facilitate this analysis later, we state two facts about the operators we just used and the covariance matrices.

\begin{proposition}\label{pro:PiPe}
\[ \PPi_1 \PP_{1,i} \ee_i = \bvec{0} . \]
\end{proposition}

\begin{proposition}\label{pro:remove_opers}
\[
\norm{ \SSigma^{(k+1)}}_2
= 
\norm{  \left( \II  - \rr \ee_i^{\top} \right) \SSigma^{(k)} 
  \left( \II  - \ee_i \rr^{\top} \right) }_2 .
\]
\end{proposition}
\begin{proof}
We have 
\[ \SSigma^{(k+1)} =  \PPi_1  \PP_{1,i} \left( \II  - \rr \ee_i^{\top} \right) \SSigma^{(k)} 
  \left( \II  - \ee_i \rr^{\top} \right) \PP_{1,i}^{\top} \PPi_1^{\top} , \]
but the operators $\PPi_1  \PP_{1,i}$ that multiply on the left and right only serve to permute the rows and columns 
  and to remove an all-zero row and column.
Neither of these operations changes the operator norm of a matrix.
\end{proof}

\section{Mathematical Tools}
In this section, we review and prove the relevant tools for establishing our main result.

\subsection{Facts about Gaussians}

We will use $\phi(x)$ and $\Phi(x)$ to denote the probability density function and the cumulative distribution function of the standard normal distribution respectively.

\begin{lemma}\label{lem:gaussian-upper-tail}
    Let $x \sim \calN(0,1)$. Then, for every $t > 0$,
    \begin{align*}
        \Pr(|x| \ge t) &\le \frac{\sqrt{2}}{t\sqrt{\pi}} \exp\left(-\frac{t^2}{2}\right).
    \end{align*}
\end{lemma}
\begin{proof}
This follows from \cite[Lemma 2 of Section VII.1]{Feller1968}, which says that for the density of the standard normal,
  $\phi(x) \colonequals (1 / \sqrt{2 \pi}) \exp (-x^2/2)$, 
  \[
  \int_{x = t}^{\infty} \phi(x) dx \leq \phi(t) / t.
  \]
\end{proof}

\begin{lemma}\label{lem:gaussian-small-ball}
     Let $x \sim \calN(0,1)$. Then, for every $t > 0$,
     \begin{align*}
        \Pr(|x| \le t) &\le t\sqrt{\frac{2}{\pi}} .
    \end{align*}
\end{lemma}

\begin{proof}
    Note that $\phi(s)$ is decreasing for $s \ge 0$. Thus,
    \begin{align*}
        \Pr(|x| \le t) = \int_{-t}^t \phi(s) ds \le 2t\phi(0) = t\sqrt{\frac{2}{\pi}}.
    \end{align*}
\end{proof}

The following lemma follows from a simple union bound and Lemma~\ref{lem:gaussian-upper-tail}.
\begin{lemma}\label{lem:gaussian-max-upper-tail}
    Let $x_1, \dots, x_n$ be independent standard normal variables. Let $M \colonequals \max_i |x_i|$. Then, for every $t > 0$,
    \begin{align*}
        \Pr(M \ge t) \le \frac{n\sqrt{2}}{t\sqrt{\pi}}\exp\left(-\frac{t^2}{2}\right).
    \end{align*}
\end{lemma}


Next, we present an important tool in our analysis, which is a comparison inequality for the probability of a conic event under a generic Gaussian distribution to its probability under an isotropic Gaussian distribution.

\begin{definition}
    Let $E \subseteq \bR^d$. We say $E$ is conic if for every $\xx \in E$ and $t \in \bR^+$, we have $t\xx \in E$.  
\end{definition}

\begin{lemma}[Comparison Inequality for Conic Event under Gaussian Measure]\label{lem:conic-prob-transport}
    Let $\SSigma$ be a $d \times d$ positive-definite  matrix, and $E \subseteq \bR^d$ be conic. Then,
    \begin{align*}
        \Pr_{\yy \sim \calN(\zzero, \SSigma)}\left(\yy \in E\right) \ge \frac{1}{\|\SSigma^{-1}\|_2^{\frac{d}{2}}\sqrt{\det(\SSigma)}} \cdot \Pr_{\yy \sim \calN(\zzero, \II_d)}\left(\yy \in E\right).
    \end{align*}
\end{lemma}

\begin{proof}
Assume for now that the smallest eigenvalue of $\SSigma$ is 1,
  which tells us that for all vectors $\yy$, $\yy^\top \SSigma^{-1} \yy \leq \norm{\yy}_2^2$.

    For a $d$-dimensional Gaussian random vector of covariance $\SSigma$
    and a measurable set $E$,
    \begin{align*}
    \Pr_{\yy \sim \calN(\zzero, \SSigma)}(\yy \in E)
&= \int_{\yy \in \bR^d} \frac{1}{(2\pi)^{\frac{d}{2}}\sqrt{\det(\SSigma)}} \exp\left(-\frac{1}{2} \yy^\top \SSigma^{-1} \yy\right) \One_{\yy \in E} \  d\yy 
    \end{align*}

Our assumption about the scaling of $\SSigma$ implies 
$\exp\left(-\frac{1}{2} \yy^\top \SSigma^{-1} \yy\right) \geq 
\exp(-\frac{1}{2} \norm{\yy}_2^2 )$, and so 
  this integral is at least 
    \begin{align*}
 \int_{\yy \in \bR^d} \frac{1}{(2\pi)^{\frac{d}{2}}\sqrt{\det(\SSigma)}} \exp\left(-\frac{1}{2}  \norm{\yy}_2^2 \right) \One_{\yy \in E} \  d\yy 
 & =
    \frac{1}{\sqrt{\det(\SSigma)} }  \Pr_{\yy \sim \calN(\zzero, \II_d)}\left(\yy \in E\right).  
\end{align*}

Finally, for general $d \times d$ positive-definite $\SSigma$, we may multiply it by $\|\SSigma^{-1}\|_2$ to obtain a positive-definite matrix with smallest eigenvalue $1$. As $E$ is conic, multiplying $\SSigma$ by a constant does not change the probability. We thus conclude 
\begin{align*}
        \Pr_{\yy \sim \calN(\zzero, \SSigma)}\left(\yy \in E\right) &\ge \frac{1}{\sqrt{\det\left(\|\SSigma^{-1}\|_2\SSigma\right)}} \cdot \Pr_{\yy \sim \calN(\zzero, \II_d)}\left(\yy \in E\right)\\
        &= \frac{1}{\|\SSigma^{-1}\|_2^{\frac{d}{2}}\sqrt{\det(\SSigma)}} \cdot \Pr_{\yy \sim \calN(\zzero, \II_d)}\left(\yy \in E\right).
    \end{align*}

\end{proof}

The following lemma is a well-known result due to Bartlett, which states that in the QR decomposition of an $n \times p$ Gaussian matrix where $p \le n$, its Q factor and R factor are independent. See \cite[Section 3.2.4]{muirhead2009aspects}.

\begin{lemma}\label{lem:gaussian-QR}
    Let $n \ge p$, and $\GG$ be an $n \times p$ matrix with independent $\calN(0,1)$ entries. Let the QR decomposition of $\GG$ be
    \begin{align*}
        \GG &= \QQ \RR,
    \end{align*}
    where $\QQ^\top\QQ = \II_p$ and $\RR$ is upper triangular with positive diagonal entries. Then, $\QQ$ and $\RR$ are independent.
\end{lemma}

\begin{definition}
For $k \in \bZ^+$, the chi-squared distribution $\chi^2_k$ with $k$ degrees of freedom is the distribution of the sum of the squares of $k$ independent standard normal variables.
\end{definition}
The moments of chi-squared random variables are as follows.

\begin{lemma}[{\cite[Equation (3.7)]{janson2010moments}}] \label{lem:chi-squared-moments}
    Let $x \sim \chi^2_k$. Then, for every $m > - \frac{k}{2}$,
    \begin{align*}
        \Ex[x^m] = 2^m \frac{\Gamma(m + \frac{k}{2})}{\Gamma(\frac{k}{2})}.
    \end{align*}
\end{lemma}

\begin{definition}
    For $p, n \in \bZ^+$ and a $p \times p$ positive semidefinite matrix $\VV$, the real Wishart distribution $W_p(\VV, n)$ is the distribution of the $p \times p$ random matrix
    \begin{align*}
        \XX = \GG\GG^\top,
    \end{align*}
    where $\GG = (\gg_1, \dots, \gg_n)$ is a $p \times n$ random matrix whose columns $\gg_i$ are independently and identically distributed according to $\calN(\zzero, \VV)$.
\end{definition}

We record the following fact about the distribution of determinants of real Wishart matrices.

\begin{lemma}[{\cite{goodman1963distribution}}] \label{lem:wishart-det}
    Let $n \ge p$. For $\XX \sim W_p(\II_p, n)$,
    \begin{align*}
        \det(\XX) \overset{d}{=} \prod_{i=0}^{p-1} x_i,
    \end{align*}
    where $x_i \sim \chi^2_{n-i}$ are independent.
\end{lemma}

\subsection{Covariance Update Tools}

\begin{proposition}\label{prop:cov-lower-bound}
    For all $k$, $\SSigma^{(k)} \succeq \II$.
\end{proposition}

\begin{proof}
    We prove this by induction.
    The base case is $k = 1$, for which we have $\SSigma^{(1)} = \II$.
    Assuming we've proved it for $k$, we establish it for $k+1$ by using the formula 
\begin{align*}
 \SSigma^{(k+1)} & = \PPi_{1} \PP_{1,i}  \left(\II_{n-k+1} - \rr\ee_{i}^\top\right)\SSigma^{(k)}\left(\II_{n-k+1} - \ee_{i}\rr^\top  \right) \PP_{1,i}^{\top} \PPi_1^{\top}  
 \\
 & \succeq \PPi_{1} \PP_{1,i}  \left(\II_{n-k+1} - \rr\ee_{i}^\top\right) \II_{n-k+1} \left(\II_{n-k+1} - \ee_{i}\rr^\top  \right) \PP_{1,i}^{\top} \PPi_1^{\top} \\
 & = \PPi_{1} \PP_{1,i} \PP_{1,i}^{\top} \PPi_1^{\top}
  + \PPi_{1} \PP_{1,i} \rr \rr^{\top} \PP_{1,i}^{\top} \PPi_1^{\top},\quad 
\text{as $\PPi_{1} \PP_{1,i} \ee_{i} = \bvec{0}$,}\\
& = \II_{n-k} + \PPi_{1} \PP_{1,i} \rr \rr^{\top} \PP_{1,i}^{\top} \PPi_1^{\top}
\\
& \succeq \II_{n-k},
\end{align*}
because $\PPi_{1} \PP_{1,i} \rr \rr^{\top} \PP_{1,i}^{\top} \PPi_1^{\top}$ is positive semidefinite.
\end{proof}

\begin{proposition}[Sufficient Condition for Covariance Growth]\label{prop:suff-cond-norm-increase}
    Fix $1 \le k \le n-1$. Let $d \colonequals n-k+1$ and $\SSigma^{(k)}$ be a $d \times d$ positive definite matrix. Let $\vv \in \bR^d$ be a top unit eigenvector of $\SSigma^{(k)}$. Let $\yy \ne \zzero \in \bR^{d}$ be any vector.
    \begin{itemize}
        \item Set $i = \argmax |\yy[j]|$.
        \item Set $\rr = \yy/\yy[i]$.
        \item Set $\SSigma^{(k+1)} = \PPi_{1} \PP_{1,i} \left(\II_{d} - \rr\ee_{i}^\top\right)\SSigma^{(k)}\left(\II_{d} - \ee_{i}\rr^\top\right) \PP_{1,i}^{\top} \PPi_1^{\top} $.
    \end{itemize}

    If for some $C \ge 1$, (1) $\vv[i]^2 \ge \frac{1}{d}$, (2) $\|\rr\|_2^2 \ge \frac{d}{C} \ge 2$, and (3) $\vv[i] \langle \vv[-i], \rr[-i] \rangle \le 0$, then
    \begin{align*}
        \|\SSigma^{(k+1)}\|_2 \ge \left(1 + \frac{1}{C} - \frac{2}{d}\right)\|\SSigma^{(k)}\|_2.
    \end{align*}

\end{proposition}

\begin{proof}
    Let $\lambda \colonequals \|\SSigma^{(k)}\|_2$. Since $\vv$ is a top eigenvector of $\SSigma^{(k)}$, we have $\SSigma^{(k)} \succeq \lambda \vv\vv^\top$. Then,
    \begin{align*}
        \|\SSigma^{(k+1)}\|_2 
        &= \left\| \PPi_{1} \PP_{1,i} \left(\II_{d} - \rr\ee_{i}^\top\right)\SSigma^{(k)}\left(\II_{d} - \ee_{i}\rr^\top\right) \PP_{1,i}^{\top} \PPi_1^{\top} \right\|_2\\
        & = \left\| \left(\II_{d} - \rr\ee_{i}^\top\right)\SSigma^{(k)}\left(\II_{d} - \ee_{i}\rr^\top\right)  \right\|_2 
        \qquad \text{(Proposition~\ref{pro:remove_opers})}\\
        &\ge \lambda \left\|\left(\II_{d} - \rr\ee_{i}^\top\right)\vv\vv^\top\left(\II_{d} - \ee_{i}\rr^\top\right)\right\|_2\\
        &= \lambda \left\|\left(\II_{d} - \rr\ee_{i}^\top\right)\vv\right\|_2^2\\
        &= \lambda \left(\|\vv\|_2^2 + \vv[i]^2 \|\rr\|_2^2 - 2\vv[i] \langle \vv, \rr \rangle\right)\\
        &= \lambda \left(\|\vv\|_2^2 + \vv[i]^2 (\|\rr\|_2^2 - 2) - 2\vv[i] \langle \vv[-i], \rr[-i] \rangle\right)\\
        &\ge \lambda \left(1 + \frac{1}{C} - \frac{2}{d}\right),
    \end{align*}
    where the last line follows from $\|\vv\|_2^2 = 1$, $\vv[i]^2 \ge \frac{1}{d}$, $\|\rr\|_2^2 \ge \frac{d}{C} \ge 2$, and $\vv[i] \langle \vv[-i], \rr[-i] \rangle \le 0$.

\end{proof}

The following lemmas and propositions provide bounds on the determinants of the covariance matrices.
Lemma~\ref{lem:cov-det} gives a formula for the determinant of the covariance in terms of the original matrix $\AA$.
Proposition~\ref{prop:det-upper-bound} tells us that the determinant of the covariance matrix is unlikely to grow too large, while Proposition~\ref{prop:det-lower-bound} says it is unlikely to become too small.
We will apply these to one matrix $\SS^{(0)}$ that appears near the end of the elimination process, when only around $\log(\rho/\sqrt{n})$ rows remain. After that, Proposition~\ref{prop:det-det-control} shows that the determinant of each subsequent covariance matrix can grow by at most a factor of the current dimension at each step. These bounds on the determinants will be useful for lower bounding the initial operator norm of the covariance matrix $\SS^{(0)}$ that we start with, and for lower bounding the probability that the operator norm that the covariance matrix grows by a constant factor stated in Proposition~\ref{prop:covariance-growth}.

\begin{lemma}[Covariance Determinant]\label{lem:cov-det}
    Fix $2 \le k \le n$. Consider the execution of GEPP on an $n \times n$ matrix $\AA$ with independent $\calN(0,1)$ entries after $k-1$ steps. Let $P \subseteq [n]$ be the set of $k-1$ row indices of $\AA$ chosen as pivots in the first $k-1$ steps by GEPP. Then, conditioning on the first $k-1$ columns of $\AA$, every column of the residual matrix $\MM^{(k)} = \AA^{(k)}[k{:}n, k{:}n]$ is independently distributed as $\calN(\zzero, \SSigma^{(k)})$, and
    \begin{align*}
        \det(\SSigma^{(k)}) &= \frac{\det\left(\AA[:,1{:}(k-1)]^\top \AA[:,1{:}(k-1)]\right)}{ \det\left(\AA[P,1{:}(k-1)]\right)^2}.
    \end{align*}
\end{lemma}

\begin{proof}
    Consider the execution of GEPP on $\AA$ after $k-1$ steps. Let $\Tilde{\AA}$ denote the matrix obtained from $\AA$ by applying the row swaps $\PP^{(1)}, \PP^{(2)}, \dots,$ $\PP^{(k-2)}, \PP^{(k-1)}$, i.e.,
    \begin{align*}
        \Tilde{\AA} \colonequals \PP^{(k-1)}\PP^{(k-2)}\dots \PP^{(2)}\PP^{(1)} \AA.
    \end{align*}
    Then, for $k \le i \le n$, we have
    \begin{align*}
        \AA^{(k)}[k{:}n, i] = \Tilde{\AA}[k{:}n,i] - \XX\Tilde{\AA}[1{:}(k-1),i],
    \end{align*}
    where $\XX \colonequals \Tilde{\AA}[k{:}n,1{:}(k-1)]\Tilde{\AA}[1{:}(k-1),1{:}(k-1)]^{-1}$. Denote $d \colonequals n - k + 1$. Since $\AA$ contains independent $\calN(0,1)$ entries, conditioning on the first $k-1$ columns of $\AA$, every column of $\MM^{(k)} = \AA^{(k)}[k{:}n,k{:}n]$ is independently distributed as
    \begin{align*}
        \calN\left(\zzero, \II_{d} + \XX\XX^\top\right).
    \end{align*}
    
    By Sylvester's identity,
    \begin{align*}
        \det(\II_d + \XX\XX^\top) &= \det\left(\II_{k-1} + \XX^\top \XX\right)\\
        &= \det\left(\YY^\top \YY\right)
    \intertext{where $\YY \colonequals \Tilde{\AA}[:,1{:}(k-1)]\Tilde{\AA}[1{:}(k-1),1{:}(k-1)]^{-1}$,}
        &= \frac{\det\left(\Tilde{\AA}[:,1{:}(k-1)]^\top \Tilde{\AA}[:,1{:}(k-1)]\right)}{\det\left(\Tilde{\AA}[1{:}(k-1),1{:}(k-1)]\right)^2}\\
        &= \frac{\det\left(\AA[:,1{:}(k-1)]^\top \AA[:,1{:}(k-1)]\right)}{ \det\left(\AA[P,1{:}(k-1)]\right)^2},
    \end{align*}
    where in the last step we used that $P \subseteq [n]$ is the set of $k-1$ row indices of $\AA$ chosen as pivots in the first $k-1$ steps, corresponding to the first $k-1$ rows of $\Tilde{\AA}$.
\end{proof}

We will apply Propositions~\ref{prop:det-upper-bound} and \ref{prop:det-lower-bound} with $d \approx \log(\rho / \sqrt{n})$.
We state them for general $d$ to facilitate their eventual sharpening.

\begin{proposition}[Determinant Upper Bound] \label{prop:det-upper-bound}
    Fix $2\le k \le n$. Let $d \colonequals n-k+1$, and $\SSigma^{(k)}$ be obtained from the covariance update process started at $\SSigma^{(1)} = \II_n$. Then, for every $t > 0$,
    \begin{align*}
        \Pr\left(\det(\SSigma^{(k)}) \ge t \right) \le \frac{4n^{d+1}}{t^{\frac{1}{3}}}
    \end{align*}
In particular,
    \begin{align*}
        \Pr\left(\det(\SSigma^{(k)}) \ge 512n^{3d+3} \right) \le \frac{1}{2}.
    \end{align*}
\end{proposition}

\begin{proof}
    Consider the execution of GEPP on an $n \times n$ matrix $\AA$ with independent $\calN(0,1)$ entries that correspond to the covariance update process. By Lemma~\ref{lem:cov-det},
    \begin{align*}
        \det(\SSigma^{(k)}) &= \frac{\det\left(\AA[:,1{:}(k-1)]^\top \AA[:,1{:}(k-1)]\right)}{ \det\left(\AA[P,1{:}(k-1)]\right)^2},
    \end{align*}
    where $P \subseteq [n]$ is the set of $k-1$ row indices of $\AA$ chosen as pivots in the first $k-1$ steps. Now, we bound the denominator and the numerator separately.

    Consider a fixed $J \subseteq [n]$ of size $k-1$. Note that $\AA[:,1{:}(k-1)]^\top \AA[:,1{:}(k-1)]$ and $\AA[J,1{:}(k-1)]^\top\AA[J,1{:}(k-1)]$ are real Wishart matrices distributed as $W_{k-1}(\II_{k-1}, n)$ and $W_{k-1}(\II_{k-1}, k-1)$ respectively. By Lemma~\ref{lem:wishart-det},
    \begin{align*}
        \det\left(\AA[:,1{:}(k-1)]^\top \AA[:,1{:}(k-1)]\right) &\sim \prod_{i=0}^{k-2} \chi_{n-i}^2,\\
        \det(\AA[J,1{:}(k-1)])^2 &\sim \prod_{i=1}^{k-1} \chi_i^2.
    \end{align*}
    By Markov's inequality, for any $z > 0$,
    \begin{align*}
        &\Pr\left(\det\left(\AA[:,1{:}(k-1)]^\top \AA[:,1{:}(k-1)]\right) \ge z\right)\le \frac{1}{z}\prod_{i=0}^{k-2}\underset{x_i \sim \chi^2_{n-i}}{\Ex}[x_i] = \frac{1}{z} \prod_{i=0}^{k-2} (n-i) = \frac{1}{z} \cdot\frac{n!}{d!}. \numberthis \label{ineq:det-numerator-bound}
    \end{align*}
    On the other hand, by a union bound over all $J \subseteq [n]$ of size $k-1$, for any $z > 0$,
    \begin{align*}
        &\Pr\left(\det\left(\AA[P,1{:}(k-1)]\right)^2 \le z\right)\\
        &\le \Pr\left(\exists\,J \subseteq [n], |J| = k-1, \text{ such that } \det\left(\AA[J,1{:}(k-1)]\right)^2 \le z\right)\\
        &\le \binom{n}{k-1} \Pr_{\substack{x_i \sim \chi^2_{i} \text{ independently},\\ \text{for } 1\le i \le k-1}}\left(\prod_{i=1}^{k-1} x_i \le z\right)\\
        &= \binom{n}{d} \underset{\substack{x_i \sim \chi^2_{i} \text{ independently},\\ \text{for } 2\le i \le k-1}}{\Ex}\left[\Pr_{x_1 \sim \chi^2_1}\left(x_1 \le \frac{z}{\prod_{i=2}^{k-1} x_i}\right)\right]\\
        &= \binom{n}{d} \underset{\substack{x_i \sim \chi^2_{i} \text{ independently},\\ \text{for } 2\le i \le k-1}}{\Ex}\left[\Pr_{y \sim \calN(0,1)}\left(|y| \le \sqrt{\frac{z}{\prod_{i=2}^{k-1} x_i}}\right) \right]\\
        &\le \binom{n}{d} \underset{\substack{x_i \sim \chi^2_{i} \text{ independently},\\ \text{for } 2\le i \le k-1}}{\Ex}\left[\sqrt{\frac{2}{\pi}}\sqrt{\frac{z}{\prod_{i=2}^{k-1} x_i}}  \right] \qquad \qquad \text{(Lemma~\ref{lem:gaussian-small-ball})}\\
        &= \binom{n}{d} \sqrt{\frac{2z}{\pi}} \prod_{i=2}^{k-1} \left(2^{-\frac{1}{2}} \frac{\Gamma(\frac{i-1}{2})}{\Gamma(\frac{i}{2})}\right) \hspace{1.37in} \text{(Lemma~\ref{lem:chi-squared-moments})}\\
        &= \binom{n}{d} \sqrt{\frac{z}{2^{k-3}\pi}} \frac{\Gamma(\frac{1}{2})}{\Gamma(\frac{k-1}{2})}.
    \end{align*}
    Since $\Gamma(\frac{1}{2}) = \sqrt{\pi}$ and
    \begin{align*}
        \Gamma\left(\frac{k-1}{2}\right) &= \begin{cases}
            \Gamma(\frac{1}{2}) \frac{(k-3)!!}{2^{((k-2)/2)}}  & \quad k \text{ even},\\
            \Gamma(1) \left(\frac{k-3}{2}\right)!  & \quad k \text{ odd},
        \end{cases} \qquad \text{(using } \Gamma(x+1)=x\Gamma(x)\text{ for } x > 0\text{)}\\
        &\ge \frac{\sqrt{(k-1)!}}{(k-1)2^{\frac{k}{2}}},
    \end{align*}
    we get
    \begin{align*}
        \Pr\left(\det\left(\AA[P,1{:}(k-1)]\right)^2 \le z\right) &\le \binom{n}{d} \sqrt{\frac{z}{2^{k-3}\pi}} \frac{\Gamma(\frac{1}{2})}{\Gamma(\frac{k-1}{2})} \le \binom{n}{d} \sqrt{\frac{8(k-1)^2z}{(k-1)!}}. \numberthis \label{ineq:det-denominator-bound}
    \end{align*}

    Combining \eqref{ineq:det-numerator-bound} and \eqref{ineq:det-denominator-bound}, we have for every $t > 0$ and $z > 0$,
    \begin{align*}
        &\Pr\left(\det(\SSigma^{(k)}) \ge t\right)\\
        &\le \Pr\left(\det\left(\AA[:,1{:}(k-1)]^\top \AA[:,1{:}(k-1)]\right) \ge tz \text{ or } \det\left(\AA[P,1{:}(k-1)]\right)^2 \le z \right)\\
        &\le \frac{1}{tz} \cdot \frac{n!}{d!} + \binom{n}{d} \sqrt{\frac{8(k-1)^2z}{(k-1)!}}.
    \end{align*}
    Setting $z = \frac{1}{t^{\frac{2}{3}}}\cdot \frac{(k-1)!}{2(k-1)^{\frac{2}{3}}}$, we get
    \begin{align*}
        \Pr\left(\det(\SSigma^{(k)}) \ge t\right) &\le \frac{4(k-1)^{\frac{2}{3}}}{t^{\frac{1}{3}}} \binom{n}{d} \le \frac{4n^{d+1}}{t^{\frac{1}{3}}}.
    \end{align*}
\end{proof}

\begin{proposition}[Determinant Lower Bound]\label{prop:det-lower-bound}
    Fix $2 \le k \le n$. Let $d \colonequals n - k + 1$, and $\SSigma^{(k)}$ be obtained from the covariance update process started at $\SSigma^{(1)} = \II_n$. Then, for every $t > 0$,
    \begin{align*}
        \Pr\left(\det(\SSigma^{(k)}) \le t \right) \le \left(\frac{t(d+2\log(n))^{d}}{n^{d(1- \frac{1}{\lceil \log(n)\rceil})}}\right)^{\lceil \log(n)\rceil}.
    \end{align*}
In particular,
    \begin{align*}
        \Pr\left(\det(\SSigma^{(k)}) \le \frac{1}{e}\left(\frac{n}{e(d + 2\log(n))}\right)^d \right) \le \frac{1}{n}.
    \end{align*}  
\end{proposition}

\begin{proof}
    Consider the execution of GEPP on an $n \times n$ matrix $\AA$ with independent $\calN(0,1)$ entries that correspond to the covariance update process. By Lemma~\ref{lem:cov-det},
    \begin{align*}
        \det(\SSigma^{(k)}) &= \frac{\det\left(\AA[:,1{:}(k-1)]^\top \AA[:,1{:}(k-1)]\right)}{ \det\left(\AA[P,1{:}(k-1)]\right)^2},
    \end{align*}
    where $P \subseteq [n]$ is the set of $k-1$ row indices of $\AA$ chosen as pivots in the first $k-1$ steps.

    Now consider a fixed $J \subseteq [n]$ of size $k-1$. Write the QR decomposition of $\AA[:,1{:}(k-1)]$ as
    \begin{align*}
        \AA[:,1{:}(k-1)] = \QQ \RR,
    \end{align*}
    where $\QQ^\top \QQ = \II_{k-1}$ and $\RR$ is upper triangular with positive diagonal entries. Note that $\det(\AA[:,1{:}(k-1)]^\top \AA[:,1{:}(k-1)]) = \det(\RR)^2$, and $\AA[J, 1{:}(k-1)] = \QQ[J,:]\RR$ for every fixed $J \subseteq [n]$ of size $k-1$. Since $\QQ$ and $\RR$ are independent by Lemma~\ref{lem:gaussian-QR},
    \begin{align*}
        \frac{\det\left(\AA[:,1{:}(k-1)]^\top \AA[:,1{:}(k-1)]\right)}{ \det\left(\AA[J,1{:}(k-1)]\right)^2} &= \frac{\det(\RR)^2}{\det(\QQ[J,:])^2 \det(\RR)^2} = \frac{1}{\det(\QQ[J,:])^2}.
    \end{align*}
    is independent of $\det(\AA[:,1{:}(k-1)]^\top \AA[:,1{:}(k-1)]) = \det(\RR)^2$. Therefore, for every fixed $J \subseteq [n]$ of size $k-1$ and every $m \in \bZ^+$, we have
    \begin{align*}
        &\Ex\left[\det(\AA[:,1{:}(k-1)]^\top \AA[:,1{:}(k-1)])^m\right] \Ex\left[\left(\frac{\det\left(\AA[:,1{:}(k-1)]^\top \AA[:,1{:}(k-1)]\right)}{ \det\left(\AA[J,1{:}(k-1)]\right)^2}\right)^{-m}\right]\\
        &= \Ex\left[\det\left(\AA[J,1{:}(k-1)]\right)^{2m}\right],
    \end{align*}
    and
    \begin{align*}
        &\Ex\left[\left(\frac{\det\left(\AA[:,1{:}(k-1)]^\top \AA[:,1{:}(k-1)]\right)}{ \det\left(\AA[J,1{:}(k-1)]\right)^2}\right)^{-m}\right]\\
        &= \frac{\Ex\left[\det\left(\AA[J,1{:}(k-1)]^\top\AA[J,1{:}(k-1)]\right)^{m}\right]}{\Ex\left[\det(\AA[:,1{:}(k-1)]^\top \AA[:,1{:}(k-1)])^m\right]}
        \intertext{Since $\AA[:,1{:}(k-1)]^\top \AA[:,1{:}(k-1)]$ and $\AA[J,1{:}(k-1)]^\top\AA[J,1{:}(k-1)]$ are real Wishart matrices distributed as $W_{k-1}(\II_{k-1}, n)$ and $W_{k-1}(\II_{k-1}, k-1)$ respectively, by Lemma~\ref{lem:wishart-det} we have}
        &= \prod_{i=0}^{k-2}\frac{\Ex_{x_i \sim \chi_{k-1-i}^2}[x_i^m]}{\Ex_{y_i \sim \chi_{n-i}^2}[y_i^m]}
        \intertext{By Lemma~\ref{lem:chi-squared-moments}, $\Ex_{x \sim \chi_k^2}[x^m] = 2^m \frac{\Gamma(m+k/2)}{\Gamma(k/2)} = \prod_{j=0}^{m-1} (k + 2j)$ for $m \in \bZ^+$. Thus,}
        &= \prod_{i=0}^{k-2} \prod_{j=0}^{m-1}\frac{k-1-i+2j}{n-i+2j}\\
        &= \prod_{j=0}^{m-1} \prod_{i=1}^d \frac{i + 2j}{n-d+i+2j} \hspace{1.3in} \text{(}d = n-k+1\text{)}\\
        &\le \prod_{j=0}^{m-1} \left(\frac{d+2j}{n+2j}\right)^d \le \left(\frac{d+2m-2}{n}\right)^{dm}.
    \end{align*}
    By Markov's inequalty, for every fixed $J \subseteq [n]$ of size $k-1$,
    \begin{align*}
        &\Pr\left(\frac{\det\left(\AA[:,1{:}(k-1)]^\top \AA[:,1{:}(k-1)]\right)}{ \det\left(\AA[J,1{:}(k-1)]\right)^2} \le t\right)\\
        &\le t^m \Ex\left[\left(\frac{\det\left(\AA[:,1{:}(k-1)]^\top \AA[:,1{:}(k-1)]\right)}{ \det\left(\AA[J,1{:}(k-1)]\right)^2}\right)^{-m}\right]\\
        &\le t^m \left(\frac{d+2m-2}{n}\right)^{dm}.
    \end{align*}
    By a union bound over all $J \subseteq [n]$ of size $k-1$, for every $t > 0$ we get
    \begin{align*}
        \Pr\left(\det(\SSigma^{(k)}) \le t\right) &= \Pr\left(\frac{\det\left(\AA[:,1{:}(k-1)]^\top \AA[:,1{:}(k-1)]\right)}{ \det\left(\AA[P,1{:}(k-1)]\right)^2} \le t\right)\\
        &\le \binom{n}{k-1} t^m \left(\frac{d+2m-2}{n}\right)^{dm}\\
        &\le \left(\frac{t(d+2(m-1))^{d}}{n^{d(1- \frac{1}{m})}}\right)^m.
    \end{align*}
    Taking $m = \lceil \log n \rceil$ produces the desired inequality.
\end{proof}

\begin{proposition}[Deterministic Determinant Control] \label{prop:det-det-control}
    Fix $1\le k \le n-1$. Let $d \colonequals n-k+1$ and $\SSigma^{(k)}$ be a $d \times d$ positive definite matrix. Let $\yy \ne \zzero \in \bR^{d}$ be any vector.
    \begin{itemize}
        \item Set $i = \argmax |\yy[j]|$.
        \item Set $\rr = \yy/\yy[i]$.
        \item Set $\SSigma^{(k+1)} = \PPi_{1} \PP_{1,i} \left(\II_{d} - \rr\ee_{i}^\top\right)\SSigma^{(k)}\left(\II_{d} - \ee_{i}\rr^\top\right) \PP_{1,i}^{\top} \PPi_1^{\top} $.
    \end{itemize}
    Then,
    \begin{align*}
        \det\left(\SSigma^{(k+1)}\right) = \det(\SSigma^{(k)}) \cdot \left( \rr^\top \left(\SSigma^{(k)}\right)^{-1} \rr\right).
    \end{align*}
    If moreover $\SSigma^{(k)} \succeq \II_d$, then
    \begin{align*}
        \det(\SSigma^{(k+1)}) \le d \cdot \det(\SSigma^{(k)}).
    \end{align*}
\end{proposition}

\begin{proof}
    If we denote $s \colonequals \SSigma^{(k)}[i,i]$, $\ssigma \colonequals \SSigma^{(k)}[-i,i]$, and $\Tilde{\SSigma} \colonequals \SSigma^{(k)}[-i,-i]$, then
    \begin{align*}
        &\det\left(\SSigma^{(k+1)}\right)\\ &= \det\left(\PPi_{1} \PP_{1,i} \left(\II_{d} - \rr\ee_{i}^\top\right)\SSigma^{(k)}\left(\II_{d} - \ee_{i}\rr^\top\right) \PP_{1,i}^{\top} \PPi_1^{\top}\right)\\
        &= \det\left(\PPi_{1} \PP_{1,i} \left(\SSigma^{(k)} - \frac{\SSigma^{(k)}[:,i]\SSigma^{(k)}[:,i]^\top}{\SSigma^{(k)}[i,i]} + \SSigma^{(k)}[i,i] \cdot \left(\frac{\SSigma^{(k)}[:,i]}{\SSigma^{(k)}[i,i]} - \rr\right) \left(\frac{\SSigma^{(k)}[:,i]}{\SSigma^{(k)}[i,i]} - \rr\right)^\top\right)\PP_{1,i}^{\top} \PPi_1^{\top}\right)\\
        &= \det\left(\Tilde{\SSigma} - \frac{\ssigma \ssigma^\top}{s} + s\left(\frac{\ssigma}{s} - \rr[-i]\right)\left(\frac{\ssigma}{s} - \rr[-i]\right)^\top\right).
    \end{align*}

    Note that $\Tilde{\SSigma} - \frac{\ssigma \ssigma^\top}{s}$ is the Schur complement of $\SSigma^{(k)}$ with respect to $\SSigma^{(k)}[i,i]$. Permuting row $i$ and column $i$ of $\SSigma^{(k)}$, we may express $\SSigma^{(k)}$ as the following block matrix
    \begin{align*}
        \SSigma^{(k)} = \begin{bmatrix}
            s & \ssigma^\top\\
            \ssigma & \Tilde{\SSigma}
        \end{bmatrix},
    \end{align*}
    and its inverse is
    \begin{align*}
        \left(\SSigma^{(k)}\right)^{-1} = \begin{bmatrix}
            \frac{1}{s} + \frac{1}{s^2} \ssigma^\top \left(\Tilde{\SSigma} - \frac{\ssigma \ssigma^\top}{s}\right)^{-1} \ssigma  & -\frac{1}{s}\ssigma^\top\left(\Tilde{\SSigma} - \frac{\ssigma \ssigma^\top}{s}\right)^{-1}\\
            -\frac{1}{s}\left(\Tilde{\SSigma} - \frac{\ssigma \ssigma^\top}{s}\right)^{-1} \ssigma & \left(\Tilde{\SSigma} - \frac{\ssigma \ssigma^\top}{s}\right)^{-1}
        \end{bmatrix}.
    \end{align*}

    Thus, we have
    \begin{align*}
        &\rr^\top \left(\SSigma^{(k)}\right)^{-1} \rr\\
        &= \begin{bmatrix}
            1 & \rr[-i]^\top
        \end{bmatrix} \begin{bmatrix}
            \frac{1}{s} + \frac{1}{s^2} \ssigma^\top \left(\Tilde{\SSigma} - \frac{\ssigma \ssigma^\top}{s}\right)^{-1} \ssigma  & -\frac{1}{s}\ssigma^\top\left(\Tilde{\SSigma} - \frac{\ssigma \ssigma^\top}{s}\right)^{-1}\\
            -\frac{1}{s}\left(\Tilde{\SSigma} - \frac{\ssigma \ssigma^\top}{s}\right)^{-1} \ssigma & \left(\Tilde{\SSigma} - \frac{\ssigma \ssigma^\top}{s}\right)^{-1}
        \end{bmatrix}\begin{bmatrix}
            1 \\ \rr[-i]
        \end{bmatrix}\\
        &= \frac{1}{s} +\left(\frac{\ssigma}{s} - \rr[-i]\right)^\top\left(\Tilde{\SSigma} - \frac{\ssigma \ssigma^\top}{s}\right)^{-1}\left(\frac{\ssigma}{s} - \rr[-i]\right).
    \end{align*}

    By the matrix determinant lemma,
    \begin{align*}
        \det\left(\SSigma^{(k+1)}\right) &= \det\left(\Tilde{\SSigma} - \frac{\ssigma \ssigma^\top}{s} + s\left(\frac{\ssigma}{s} - \rr[-i]\right)\left(\frac{\ssigma}{s} - \rr[-i]\right)^\top\right)\\
        &= \det\left(\Tilde{\SSigma} - \frac{\ssigma \ssigma^\top}{s}\right)\left(1 + s\left(\frac{\ssigma}{s} - \rr[-i]\right)^\top \left(\Tilde{\SSigma} - \frac{\ssigma \ssigma^\top}{s}\right)^{-1} \left(\frac{\ssigma}{s} - \rr[-i]\right)\right)\\
        &= \det\left(\Tilde{\SSigma} - \frac{\ssigma \ssigma^\top}{s}\right) \cdot \left(s \rr^\top \left(\SSigma^{(k)}\right)^{-1} \rr\right).
    \end{align*}

    By Schur's formula,
    \begin{align*}
        \det\left(\Tilde{\SSigma} - \frac{\ssigma \ssigma^\top}{s}\right) = \frac{\det(\SSigma^{(k)})}{s},
    \end{align*}
    and we get
    \begin{align*}
        \det\left(\SSigma^{(k+1)}\right) = \det(\SSigma^{(k)}) \cdot \left( \rr^\top \left(\SSigma^{(k)}\right)^{-1} \rr\right).
    \end{align*}
    
    Since $\|\rr\|_{\infty} \le 1$, we have $\|\rr\|_2^2 \le d$. Hence, if $\SSigma^{(k)} \succeq \II_d$, then 
    \begin{align*}
        \det(\SSigma^{(k+1)}) &= \det(\SSigma^{(k)}) \cdot \left( \rr^\top \left(\SSigma^{(k)}\right)^{-1} \rr\right)\\
        &\le \det(\SSigma^{(k)}) \cdot \|\rr\|_2^2\\
        &\le d \cdot \det(\SSigma^{(k)}).
    \end{align*}
\end{proof}

\section{Proof of Main Result}

We are now ready to give the full proof of Theorem~\ref{thm:main}. We break the proof into a few propositions.

\begin{proposition}\label{prop:isotropic-prob}
    Fix $d \ge 4$. Let $\vv \in \bR^d$ be a unit vector. Define $E \subseteq \bR^d$ as
    \begin{align*}
        E &\colonequals \Bigg\{\yy \in \bR^d \bigg\vert \vv[i]^2 \ge \frac{1}{d}, \|\rr\|_2^2 \ge \frac{d}{2}, \text{and } \vv[i] \langle \vv[-i], \rr[-i] \rangle \le 0,\\
        &\quad \text{ where } i \colonequals \argmax_{1 \leq j \leq d} |\yy[j]|, \rr \colonequals \frac{\yy}{\yy[i]} \Bigg\}.
    \end{align*}

    Then,
    \begin{align*}
        \Pr_{\yy \sim \calN(\zzero, \II_d)}\left(\yy \in E\right) \ge \frac{0.16^d}{2d}.
    \end{align*}
\end{proposition}

\begin{proof}
    There exists at least one $i^* \in [d]$ such that $\vv[i^*]^2 \ge \frac{1}{d}$ as $\vv \in \bR^d$ is a unit vector. 
    Fix one such $i^*$. Since the distribution of $i = \argmax |\yy[j]|$ is uniform over $[d]$ for $\yy \sim \calN(\zzero, \II_d)$, we have
    \begin{align*}
        \Pr_{\yy \sim \calN(\zzero, \II_d)}\left(i = i^*\right) = \frac{1}{d}.
    \end{align*}
    For every $j \in [d]$, we have
    \begin{align*}
        \Pr_{y_j \sim \calN(0,1)}\left(\frac{1}{2} \le y_j^2 \le 1\right) &= 2\left(\Phi(1) - \Phi(1/\sqrt{2})\right) \ge 0.16.
    \end{align*}
    Thus,
    \begin{align*}
        \Pr_{\yy \sim \calN(\zzero, \II_d)}\left(\|\rr\|_2^2 \ge \frac{d}{2}\right) &\ge \Pr_{\yy \sim \calN(\zzero, \II_d)}\left(\frac{1}{2} \le y_j^2 \le 1, \text{ for all } j \in [d]\right)\\
        &= \Pr_{x \sim \calN(0, 1)}\left(\frac{1}{2} \le x^2 \le 1\right)^d\\
        &\ge 0.16^d.
    \end{align*}
    Note that for $\rr = \yy/\yy[i]$ where $\yy \sim \calN(\zzero, \II_d)$, the distribution of $\|\rr\|_2^2$ is independent of $i = \argmax |\yy[j]|$. Thus,
    \begin{align*}
        \Pr_{\yy \sim \calN(\zzero, \II_d)}\left(i = i^*, \|\rr\|_2^2 \ge \frac{d}{2}\right) &= \Pr_{\yy \sim \calN(\zzero, \II_d)}\left(i = i^*\right)\Pr_{\yy \sim \calN(\zzero, \II_d)}\left(\|\rr\|_2^2 \ge \frac{d}{2}\right)\\
        &\ge \frac{0.16^d}{d}.
    \end{align*}
    
    Finally, conditional on the event that $i = i^*$ and $\|\rr\|_2^2 \ge \frac{d}{2}$, the conditional distribution of $\rr[-i] = \yy[-i] / \yy[i]$ is symmetric with respect to the origin. As a result,
    \begin{align*}
        \Pr_{\yy \sim \calN(\zzero, \II_d)}\left( \vv[i] \langle \vv[-i], \rr[-i] \rangle \le 0\bigg\vert i = i^*, \|\rr\|_2^2 \ge \frac{d}{2}\right) \ge \frac{1}{2},
    \end{align*}
    and we get
    \begin{align*}
        \Pr_{\yy \sim \calN(\zzero, \II_d)}\left(\yy \in E\right) &\ge \Pr_{\yy \sim \calN(\zzero, \II_d)}\left(i = i^*, \|\rr\|_2^2 \ge \frac{d}{2}\right)\Pr_{\yy \sim \calN(\zzero, \II_d)}\left( \vv[i] \langle \vv[-i], \rr[-i] \rangle \le 0\bigg\vert i = i^*, \|\rr\|_2^2 \ge \frac{d}{2}\right)\\
        &\ge \frac{0.16^d}{2d}.
    \end{align*}
\end{proof}

\begin{proposition}\label{prop:covariance-growth}
    Fix $1 \le k \le n-3$. Let $d \colonequals n-k+1$ and $\SSigma^{(k)} \succeq \II_{d}$ be a $d \times d$ positive-definite matrix. Let $\SSigma^{(k+1)} \in \bR^{(d-1) \times (d-1)}$ be obtained from $\SSigma^{(k)}$ according to the covariance update process as in Proposition~\ref{pro:covariance_update}. Then,
    \begin{align*}
        \Pr\left(\|\SSigma^{(k+1)}\|_2 \ge \left(\frac{3}{2} - \frac{2}{d}\right)\|\SSigma^{(k)}\|_2 \right) \ge \frac{0.16^d}{2d \cdot \sqrt{\det(\SSigma^{(k)})}}.
    \end{align*}
\end{proposition}

\begin{proof}
    Recall that to generate $\SSigma^{(k+1)}$ from $\SSigma^{(k)}$ according to the covariance update process, we sample $\yy \sim \calN(\zzero, \SSigma^{(k)})$, and set $\SSigma^{(k+1)} = \PPi_1  \PP_{1,i} \left( \II  - \rr \ee_i^{\top} \right) \SSigma^{(k)} 
  \left( \II  - \ee_i \rr^{\top} \right) \PP_{1,i}^{\top} \PPi_1^{\top}$, where $i = \argmax |\yy[j]|$ and $\rr = \yy/\yy[i]$.

    Let $\vv \in \bR^d$ be a top unit eigenvector of $\SSigma^{(k)}$. Define $E \subseteq \bR^d$ as in Proposition~\ref{prop:isotropic-prob},
    \begin{align*}
        E &\colonequals \Bigg\{\yy \in \bR^d \bigg\vert \vv[i]^2 \ge \frac{1}{d}, \|\rr\|_2^2 \ge \frac{d}{2}, \text{and } \vv[i] \langle \vv[-i], \rr[-i] \rangle \le 0,\\
        &\quad \text{ where } i \colonequals \argmax |\yy[j]|, \rr \colonequals \frac{\yy}{\yy[i]} \Bigg\}.
    \end{align*}
    Observe that $E$ is conic and $\left\|\left(\SSigma^{(k)}\right)^{-1}\right\|_2 \le 1$. By Lemma~\ref{lem:conic-prob-transport} and Proposition~\ref{prop:isotropic-prob}, we have
    \begin{align*}
        \Pr_{\yy \sim \calN(\zzero, \SSigma^{(k)})}\left(\yy \in E\right) &\ge \frac{1}{\left\|\left(\SSigma^{(k)}\right)^{-1}\right\|_2^{\frac{d}{2}}\sqrt{\det(\SSigma^{(k)})}}\cdot\Pr_{\yy \sim \calN(\zzero, \II_d)}\left(\yy \in E\right) \ge \frac{0.16^d}{2d \cdot \sqrt{\det(\SSigma^{(k)})}}.
    \end{align*}
    Finally, by Proposition~\ref{prop:suff-cond-norm-increase}, whenever $\yy \in E$, $\|\SSigma^{(k+1)}\|_2 \ge \left(\frac{3}{2} - \frac{2}{d}\right)\|\SSigma^{(k)}\|_2$. This finishes the proof that 
    \begin{align*}
        \Pr\left(\|\SSigma^{(k+1)}\|_2 \ge \left(\frac{3}{2} - \frac{2}{d}\right)\|\SSigma^{(k)}\|_2 \right) \ge \frac{0.16^d}{2d \cdot \sqrt{\det(\SSigma^{(k)})}}.
    \end{align*}
\end{proof}

\begin{proposition}\label{prop:large-operator-norm-bound}
    Fix $T \ge 4$. Let $T \le d \le n-2$. Consider the covariance update process that starts with $\SSigma^{(1)} = \II_n$. Let $r \colonequals d - T$ and $\SS^{(i)} \colonequals \SSigma^{(n-d+i)}$ for $0 \le i \le r$. Then,
    \begin{align*}
        \Pr\left(\frac{\|\SS^{(r)}\|_2}{\max\{N_d,1\}} \ge \left(\frac{3}{2} - \frac{2}{T}\right)^r \right) \ge \left(\frac{1}{2} - \frac{1}{n}\right) \cdot \frac{0.16^{\frac{r(r+2T+3)}{2}}}{(32\sqrt{2})^{r}(r+T+1)^{\frac{r(r+3)}{4}}n^{\frac{3r(r+T+2)}{2}}},
    \end{align*}
    where we define
    \[N_d \colonequals \frac{n}{e^{1 + \frac{1}{d+1}}(d+1 + 2\log(n))}.\]
\end{proposition}
\begin{proof}
    Let $g \colonequals \left(\frac{3}{2} - \frac{2}{T}\right)^r$. For $0 \le i \le r-1$, define the event 
    \[E_i \colonequals \left\{\|\SS^{(i+1)}\|_2 \ge \left(\frac{3}{2} - \frac{2}{T}\right)\|\SS^{(i)}\|_2\right\}.\]
    We note that
    \begin{align*}
        \bigcap_{i=0}^{r-1} E_i &= \bigcap_{i=0}^{r-1} \left\{\|\SS^{(i+1)}\|_2 \ge \left(\frac{3}{2} - \frac{2}{T}\right)\|\SS^{(i)}\|_2\right\}\\
        &\subseteq \left\{\|\SS^{(r)}\|_2 \ge g\|\SS^{(0)}\|_2\right\},
    \end{align*}
    where we use that $\|\SS^{(i+1)}\|_2 \ge \left(\frac{3}{2} - \frac{2}{T}\right)\|\SS^{(i)}\|_2$ for all $0 \le i \le r-1$ implies that $\|\SS^{(r)}\|_2 \ge \left(\frac{3}{2} - \frac{2}{T}\right)^r \|\SS^{(0)}\|_2 =  g\|\SS^{(0)}\|_2$. Additionally, we define the event
    \begin{align*}
        D \colonequals \left\{\frac{1}{e} \left(\frac{n}{e(d+1 + 2\log(n))}\right)^{d+1}\le \det(\SS^{(0)}) \le 512n^{3d+6}\right\}.
    \end{align*}
    Since $\SS^{(0)}$ is a $(d+1) \times (d+1)$ positive definite matrix, $\|\SS^{(0)}\|_2 \ge \det(\SS^{(0)})^{\frac{1}{d+1}}$ deterministically. Conditioning on the event $D$, it holds that
    $\|\SS^{(0)}\|_2 \ge N_d$, where $N_d \colonequals n/(e^{1 + \frac{1}{d+1}}(d +1 + 2\log(n)))$. By Proposition~\ref{prop:cov-lower-bound}, $\|\SS^{(0)}\|_2 \ge 1$.
    
    Therefore, we have
    \begin{align*}
        &\Pr\left(\|\SS^{(r)}\|_2 \ge g \cdot \max\{N_d,1\}\right)\\
        &\ge \Pr\left(\left(\bigcap_{i=0}^{r-1} E_i\right) \bigcap D\right)\\
        &= \Pr(D)\prod_{i=0}^{r-1}\Pr\left(E_{i} \,\bigg\vert\,\left(\bigcap_{j=0}^{i-1} E_j\right)\bigcap D\right). \numberthis \label{ineq:growth-lower-bound}
    \end{align*}
    
    Since $\det(\SS^{(i)})\le \det(\SS^{(0)}) \prod_{j=0}^{i-1} (d+1-j)$ by Proposition~\ref{prop:det-det-control}, and $\SS^{(i)} \succeq \II_{d+1-i}$ by Proposition~\ref{prop:cov-lower-bound}, each conditional probability can be bounded by
    \begin{align*}
        &\Pr\left(E_{i} \,\bigg\vert\,\left(\bigcap_{j=0}^{i-1} E_j\right)\bigcap D\right)\\
        &= \Pr\left(\|\SS^{(i+1)}\|_2 \ge \left(\frac{3}{2} - \frac{2}{T}\right)\|\SS^{(i)}\|_2\,\bigg\vert\,\left(\bigcap_{j=0}^{i-1} E_j\right) \bigcap D\right)\\
        &\ge \Ex\left[\frac{0.16^{d+1-i}}{2(d+1-i)\sqrt{\det(\SS^{(i)})} } \,\bigg\vert\,\left(\bigcap_{j=0}^{i-1} E_j\right) \bigcap D\right] \qquad\qquad \text{(Proposition~\ref{prop:covariance-growth})}\\
        &\ge \frac{0.16^{d+1-i}}{2(d+1-i)\sqrt{512 n^{3d+6} \prod_{j=0}^{i-1}(d+1-j)} }\\
        &\ge \frac{0.16^{d+1-i}}{32\sqrt{2}(d+1)^{\frac{i}{2}+1} n^{\frac{3}{2}d+3}  }.
    \end{align*}
    Plugging this bound back to \eqref{ineq:growth-lower-bound}, we get
    \begin{align*}
        &\Pr\left(\|\SS^{(r)}\|_2 \ge g \cdot \max\{N_d,1\}\right)\\
        &\ge \Pr(D)\prod_{i=0}^{r-1}\Pr\left(E_{i} \,\bigg\vert\,\left(\bigcap_{j=0}^{i-1} E_j\right)\bigcap D\right)\\
        &\ge \Pr(D)\prod_{i=0}^{r-1} \frac{0.16^{d+1-i}}{32\sqrt{2}(d+1)^{\frac{i}{2}+1} n^{\frac{3}{2}d+3}}\\
        &\ge \Pr(D)\frac{0.16^{\frac{r(r+2T+3)}{2}}}{(32\sqrt{2})^{r}(r+T+1)^{\frac{r(r+3)}{4}}n^{\frac{3r(r+T+2)}{2}}}.
    \end{align*}
    Finally, by Proposition~\ref{prop:det-upper-bound} and Proposition~\ref{prop:det-lower-bound},
    \begin{align*}
        \Pr(D) &\ge 1 - \Pr\left(\det(\SS^{(0)}) \ge 512n^{3d+6}\right) - \Pr\left(\det(\SS^{(0)}) \le \frac{1}{e} \left(\frac{n}{e(d+1 + 2\log(n))}\right)^{d+1}\right)\\
        &\ge 1 - \frac{1}{2} - \frac{1}{n} =  \frac{1}{2} - \frac{1}{n},
    \end{align*}
    and we conclude that
    \begin{align*}
        \Pr\left(\|\SS^{(r)}\|_2 \ge g \cdot \max\{N_d,1\}\right) \ge \left(\frac{1}{2} - \frac{1}{n}\right) \cdot \frac{0.16^{\frac{r(r+2T+3)}{2}}}{(32\sqrt{2})^{r}(r+T+1)^{\frac{r(r+3)}{4}}n^{\frac{3r(r+T+2)}{2}}}.
    \end{align*}
\end{proof}

\begin{proposition}\label{prop:operator-norm-to-entry-magnitude}
    Fix $1 \le d \le n$. 
    For every $g > 0$, if $\UU$ is the upper triangular matrix produced by GEPP when run on a matrix of independent $\calN(0,1)$ variables, and $\SSigma^{(k)}$ is the covariance matrix of the first column of $\MM^{(k)}$, then
    \begin{align*}
        \Pr\left(\max_{i,j} \left|\UU[i,j]\right| \ge \sqrt{\frac{g}{2d}} \,\bigg \vert\, \|\SSigma^{(n-d+1)}\|_2 \ge g\right) \ge 0.47,
    \end{align*}
    provided that $\Pr\left(\|\SSigma^{(n-d+1)}\|_2 \ge g\right) > 0$.
\end{proposition}

\begin{proof}
    Condition on the event that $\|\SSigma^{(n-d+1)}\|_2 \ge g$. Let $\vv$ be a top unit eigenvector of $\SSigma^{(n-d+1)}$. Note that in the covariance update process, $\SSigma^{(n-d+1)}$ is the covariance matrix of all the columns of $\MM^{(n-d+1)}$. In particular, conditional on the event that $\|\SSigma^{(n-d+1)}\|_2 \ge g$ and the history up to step $n-d$, the first column $\mm \colonequals \MM^{(n-d+1)}[:,1]$ of $\MM^{(n-d+1)}$ satisfies that
    \begin{align*}
        \Pr_{\mm \sim \calN(\zzero, \SSigma^{(n-d+1)})}\left(\|\mm\|_2^2 \ge \frac{g}{2} \right) &\ge \Pr_{\mm \sim \calN(\zzero, \SSigma^{(n-d+1)})}\left(\langle\vv, \mm \rangle^2 \ge \frac{g}{2}\right)\\
        &= \Pr_{x \sim \calN(\zzero, \|\SSigma^{(n-d+1)}\|_2)}\left(x^2 \ge \frac{g}{2}\right)\\
        &= 2\left(1 - \Phi\left(\sqrt{\frac{g}{2\|\SSigma^{(n-d+1)}\|_2}}\right)\right)\\
        &\ge 2\left(1 - \Phi\left(1/\sqrt{2}\right)\right)\\
        &\ge 0.47.
    \end{align*}
    Once $\|\mm\|_2^2 \ge \frac{g}{2}$, we know that at least one entry of $\mm \in \bR^{d}$ has absolute value at least $\frac{\|m\|_2}{\sqrt{d}} \ge  \sqrt{\frac{g}{2d}}$. Moreover, the entry of $\mm$ with the largest absolute value is chosen as the pivot and appears at the diagonal entry $\UU[n-d+1,n-d+1]$. Therefore,
    \begin{align*}
        \Pr\left(\max_{i,j} \left|\UU[i,j]\right| \ge \sqrt{\frac{g}{2d}} \,\bigg \vert\, \|\SSigma^{(n-d+1)}\|_2 \ge g\right) \ge 0.47.
    \end{align*}
\end{proof}

\subsection{Proof of Theorem~\ref{thm:main}} \label{sec:proof-main}
We now prove our main theorem. Consider the covariance update process that starts with $\SSigma^{(1)} = \II_n$. Fix $T \ge 4$. Let $T \le d \le n-2$ and $r \colonequals d - T$. Consider the event
\[G_r \colonequals \left\{\|\SSigma^{(n-d+r)}\|_2 \ge \left(\frac{3}{2} - \frac{2}{T}\right)^r \max\{N_d,1\}\right\},\]
where $N_d \colonequals \frac{n}{e^{1 + \frac{1}{d+1}}(d+1 + 2\log(n))}$. By Proposition~\ref{prop:large-operator-norm-bound},
\begin{align*}
        \Pr\left(G_r\right) \ge \left(\frac{1}{2} - \frac{1}{n}\right) \cdot \frac{0.16^{\frac{r(r+2T+3)}{2}}}{(32\sqrt{2})^{r}(r+T+1)^{\frac{r(r+3)}{4}}n^{\frac{3r(r+T+2)}{2}}}.
    \end{align*}
By Proposition~\ref{prop:operator-norm-to-entry-magnitude},
\begin{align*}
        \Pr\left(\max_{i,j} \left|\UU[i,j]\right| \ge \sqrt{\frac{\left(\frac{3}{2} - \frac{2}{T}\right)^r \max\{N_d,1\}}{2(T+1)}} \,\Bigg \vert\, G_r\right) \ge 0.47.
    \end{align*}
    Therefore, we obtain the following guarantee on the numerator of the growth factor:
    \begin{align*}
        &\Pr\left(\max_{i,j} \left|\UU[i,j]\right| \ge \sqrt{\frac{\left(\frac{3}{2} - \frac{2}{T}\right)^r \max\{N_d,1\}}{2(T+1)}}\right)\\
        &\ge \Pr\left(\max_{i,j} \left|\UU[i,j]\right| \ge \sqrt{\frac{\left(\frac{3}{2} - \frac{2}{T}\right)^r \max\{N_d,1\}}{2(T+1)}} \,\Bigg\vert\,G_r\right)\Pr\left(G_r\right)\\
        &\ge 0.235 \left(1 - \frac{2}{n}\right) \cdot \frac{0.16^{\frac{r(r+2T+3)}{2}}}{(32\sqrt{2})^{r}(r+T+1)^{\frac{r(r+3)}{4}}n^{\frac{3r(r+T+2)}{2}}}. \numberthis \label{ineq:growth-numerator}
    \end{align*}
    
    On the other hand, we may upper bound the denominator of the growth factor using Lemma~\ref{lem:gaussian-max-upper-tail}:
    \begin{align*}
        \Pr\left(\max_{i,j}|\AA[i,j]| \ge s\right) \le \frac{n^2 \sqrt{2}}{s \sqrt{\pi}}\exp\left(-\frac{s^2}{2}\right), \quad\forall\, s > 0.
    \end{align*}
    Setting $s = 3r\sqrt{\log n}$, we get
    \begin{align*}
        \Pr\left(\max_{i,j}|\AA[i,j]| \ge 3r\sqrt{\log n}\right) &\le \frac{n^2 \sqrt{2}}{3r\sqrt{\log n} \sqrt{\pi}}\exp\left(-\frac{9r^2 \log(n)}{2}\right)\\
        &\le \frac{\sqrt{2}}{3r\sqrt{\pi \log n}} \cdot \frac{1}{n^{\frac{9}{2}r^2 - 2}}. \numberthis \label{ineq:growth-denominator}
    \end{align*}

    Combining the bounds \eqref{ineq:growth-numerator} and \eqref{ineq:growth-denominator}, if we set $T = 20$, we get
    \begin{align*}
        &\Pr\left(\rho(\AA) \ge \frac{\sqrt{\frac{(7/5)^r \max\{N_{r + 20},1\}}{42}}}{3r\sqrt{\log n}}\right)\\
        &\ge \Pr\left(\max_{i,j} \left|\UU[i,j]\right| \ge \sqrt{\frac{\left(\frac{7}{5}\right)^r \max\{N_{r + 20},1\}}{42}}\right)- \Pr\left(\max_{i,j}|\AA[i,j]| \ge 3r\sqrt{\log n}\right)\\
        &\ge 0.235\left(1 - \frac{2}{n}\right) \cdot \frac{0.16^{\frac{r(r+43)}{2}}}{(32\sqrt{2})^{r}(r+21)^{\frac{r(r+3)}{4}}n^{\frac{3r(r+22)}{2}}} - \frac{\sqrt{2}}{3r\sqrt{\pi \log n}} \cdot \frac{1}{n^{\frac{9}{2}r^2 - 2}}\\
        &= \exp\Bigg(\frac{r(r+43)}{2}\log(0.16) - r\log(32\sqrt{2}) - \frac{r(r+3)}{4}\log(r+21) - \frac{3r(r+22)}{2}\log(n)\\
        &\hspace{0.3in}+ \log\left(0.235\left(1 - \frac{2}{n}\right)\right)\Bigg) - \exp\left(\log\left(\frac{\sqrt{2}}{3\sqrt{\pi}}\right) - \left(\frac{9}{2}r^2 - 2\right)\log(n) -\log(r) -\frac{1}{2}\log\log(n)\right)\\
        &\ge \exp\left(-\frac{r(7r + 135)}{4}\log(n) -0.917r^2 - 43.213r -1.4482 +\log\left(1- \frac{2}{n}\right)\right) -\exp\left(-\left(\frac{9}{2} r^2-2 \right)\log(n)\right)\\
        &\ge \exp\left(-2.369r^2\log(n)\right) - \exp\left(-4.499r^2\log(n)\right)\\
        &\ge \exp\left(-2.37r^2\log(n)\right),
    \end{align*}
    which holds for all $r \ge 100$ and $n \ge r + 22$.

    Now define the function \[f(r) \colonequals \frac{\sqrt{\frac{(7/5)^r \max\{N_{r + 20},1\}}{42}}}{3r\sqrt{\log n}} = \frac{\sqrt{\left(\frac{7}{5}\right)^r} \max\left\{\frac{\sqrt{n}}{\sqrt{e^{1 + \frac{1}{r+21}} (r+21 + 2\log(n))}} , 1\right\}}{3\sqrt{42} \cdot r \sqrt{\log(n)}},\] which is increasing for $r \ge 10$. Let $\rho_{\min} \colonequals f(100)$, and $\rho_{\max} \colonequals f(n-22)$. Now, for any $\rho \in [\rho_{\min}, \rho_{\max}]$, define \[r_\rho \colonequals \min \{r \in \bN: f(r) \ge \rho, r\ge 10\}.\] Then, we have
    \begin{align*}
        \Pr\left(\rho(\AA) \ge \rho\right) &\ge \Pr\left(\rho(\AA) \ge f(r_{\rho})\right) \ge \exp\left(-2.37r_{\rho}^2\log(n)\right).
    \end{align*}
    Finally, assume that $n \ge 122$, then $100\sqrt{n}\log(n) > f(100) = \rho_{\min}$, and $\frac{\left(\frac{7}{5}\right)^{\frac{n-22}{2}}}{3\sqrt{42} (n-22)\sqrt{\log(n)}} \le f(n-22) = \rho_{\max}$. For any $\rho \in \left[100\sqrt{n}\log(n), \frac{\left(\frac{7}{5}\right)^{\frac{n-22}{2}}}{3\sqrt{42} (n-22)\sqrt{\log(n)}}\right]$, we have
    \begin{align*}
        f(r_\rho - 1) \le \rho \le f(r_{\rho}),
    \end{align*}
    where $101 \le r_{\rho} \le n-22$. Taking the logarithm and using that for $n \ge 122$ and $101 \le r_{\rho} \le n-22$,
    \[e^{1 + \frac{1}{r_{\rho}+20}} (r_{\rho}+20 + 2\log(n)) \le 0.749(r_{\rho} - 1)\log(n),\]
    we get
    \begin{align*}
        \log(\rho) &\ge \log f(r_{\rho} - 1)\\
        &\ge \frac{r_{\rho} - 1}{2}\log(7/5) + \frac{1}{2}\log(n) -\frac{1}{2}\log(0.749) - \frac{1}{2}\log(r_{\rho}-1) - \frac{1}{2}\log\log(n)\\
        &\hspace{0.4in}- \frac{1}{2}\log(42) - \log(3(r_{\rho}-1)) - \frac{1}{2}\log\log n\\
        &\ge 0.168(r_{\rho} - 1) + \frac{1}{2}\log(n) - \frac{3}{2}\log(r_{\rho} - 1) -\log\log(n) + 0.144   -  1.869 - 1.098.
    \end{align*}
    Rearranging the inequality and using that $\log(\rho) \ge \log(100) + \frac{1}{2}\log(n) + \log\log(n)$, we get
    \begin{align*}
        \log\left(\frac{\rho}{\sqrt{n}}\right) &\ge 0.084(r_{\rho} - 1) - \frac{3}{4}\log(r_{\rho} - 1) + 0.891\\
        &\ge 0.05779 r_{\rho},
    \end{align*}
    which holds for $r_{\rho} \ge 101$.

    Thus, we conclude that for $n \ge 122$, and any $$\rho \in \left[100\sqrt{n}\log(n),\; \frac{\left(\frac{7}{5}\right)^{\frac{n-22}{2}}}{3\sqrt{42} (n-22)\sqrt{\log(n)}}\right],$$
    we have
    \begin{align*}
        \Pr\left(\rho(\AA) \ge \rho\right) &\ge \exp\left(-2.37r_{\rho}^2\log(n)\right)\\
        &\ge \exp\left(-710 \log^2\left(\frac{\rho }{\sqrt{n}}\right)\log(n)\right).
    \end{align*}

\bibliographystyle{alpha}
\bibliography{main}

\end{document}